\documentclass[11pt,letterpaper]{amsart}
\usepackage{amsmath, amsthm, amsfonts, amssymb}
\usepackage{enumerate}
\usepackage[bookmarks=false]{hyperref}
\hypersetup{breaklinks=true}

\usepackage{xcolor}

\usepackage{verbatim}

\usepackage[T1]{fontenc}

\theoremstyle{plain}
\newtheorem{main}{Theorem}

\newtheorem{thm}{Theorem}[section]
\newtheorem*{thm*}{Theorem}
\newtheorem{lem}[thm]{Lemma}
\newtheorem{fact}[thm]{Fact}

\newtheorem{prop}[thm]{Proposition}
\newtheorem*{prop*}{Proposition}

\newtheorem{cor}[thm]{Corollary}
\newtheorem*{cor*}{Corollary}

\theoremstyle{definition}
\newtheorem{defn}[thm]{Definition}
\newtheorem*{defn*}{Definition}

\newtheorem{remark}[thm]{Remark}
\newtheorem{remarks}[thm]{Remarks}
\newtheorem{question}[thm]{Question}

\newtheorem*{question*}{Question}
\newtheorem*{Pquestion*}{Popa's question}

\newtheorem*{conv*}{Convention}

\newcommand{\dminus}{ 
\buildrel\textstyle\ .\over{\hbox{ 
\vrule height3pt depth0pt width0pt}{\smash-} 
}}

\newcommand{\norm}[1]{{\left\lVert #1\right\rVert}}

\def\bb{\mathbb}

\def\bb{\mathbb}

\def\cal{\mathcal}

\def\u{\mathsf 1}

\def\dotminussym#1#2{%
  \setbox0=\hbox{$\m@th#1-$}%
  \kern.5\wd0%
  \hbox to 0pt{\hss\hbox{$\m@th#1-$}\hss}%
  \raise.6\ht0\hbox to 0pt{\hss$\m@th#1.$\hss}%
  \kern.5\wd0}

\DeclareMathOperator{\id}{id}

\DeclareMathOperator{\tr}{tr}

\newcommand\bC{{\mathbb C}}

\def \R{\mathcal R}
\def \u{\mathcal U}

\usepackage[
backend=biber,
style=alphabetic, url=false
]{biblatex}
\begin{document}

%%%%%%%%%%%%%%%%%%%%%%%%%%%%%%%%%%%%%%%%%%%%%%

\title{Uniformly Super McDuff II$_1$ Factors}
\author[Goldbring, Jekel, Kunnawalkam Elayavalli, and Pi]{Isaac Goldbring, David Jekel, Srivatsav Kunnawalkam Elayavalli, and Jennifer Pi}

\thanks{I. Goldbring was partially supported by NSF grant DMS-2504477.  D. Jekel was partially supported by NSF grant DMS-2002826. S. Kunnawalkam Elayavalli was supported by a Simons Postdoctoral Fellowship.}

\address{Department of Mathematics\\University of California, Irvine, 340 Rowland Hall (Bldg.\# 400),
Irvine, CA 92697-3875}
\email{isaac@math.uci.edu}
\urladdr{http://www.math.uci.edu/~isaac}

\address{Department of Mathematics\\University of California, Irvine, 410 Rowland Hall (Bldg.\# 400),
Irvine, CA 92697-3875}
\email{jspi@math.uci.edu}
\urladdr{https://sites.uci.edu/jpi314/}

\address{Department of Mathematics\\University of California, San Diego, 9500 Gilman Dr,
La Jolla, CA 92093}
\email{djekel@ucsd.edu}
\urladdr{https://davidjekel.com/}

\address{Institute of Pure and Applied Mathematics, UCLA, 460 Portola Plaza, Los Angeles, CA 90095, USA}\email{srivatsav.kunnawalkam.elayavalli@vanderbilt.edu}
\urladdr{https://sites.google.com/view/srivatsavke/home}

\begin{abstract}

We introduce and study the family of uniformly super McDuff II$_1$ factors. This family is shown to be closed under elementary equivalence  and  also coincides with the family of II$_1$ factors with the Brown property introduced in \cite{AGKE-GeneralizedJung}.  We  show that a certain family of existentially closed factors, the so-called infinitely generic factors, are uniformly super McDuff, thereby  improving  a recent result of \cite{ChifanDI-EmbUnivPptyT}.  We also show that Popa's family of strongly McDuff II$_1$ factors are uniformly super McDuff.  Lastly, we investigate when finitely generic II$_1$ factors are uniformly super McDuff.         

%We study the problem of determining whether or not the Brown and super McDuff properties of separable II$_1$ factors coincide, equivalently, whether or not being super McDuff is closed under elementary equivalence.  We show that the Brown property is equivalent to a uniform version of the super McDuff property and use this to show that the class of infinitely generic II$_1$ factors have the Brown property. We conclude by documenting some observations about the finitely generic factors in this context.       
\end{abstract}
\maketitle

\section{Introduction}

The study of central sequences in II$_1$ factors has been a central theme in operator algebras since the days of Murray and von Neumann, who introduced in \cite{MvN43} the notion of Property Gamma which states that nontrivial central sequences exist and which can be used to distinguish the hyperfinite II$_1$ factor from the free group factor $L(\mathbb{F}_2)$.  McDuff \cite{McD1} leveraged the property of having abundantly many central sequences  to construct continuum many non-isomorphic separable II$_1$ factors, settling a fundamental question open since the inception of the subject.  In \cite{McD2}, McDuff showed that the central sequence algebra being noncommutative is equivalent to tensorially absorbing the hyperfinite II$_1$ factor; ever since then, this strengthening of property Gamma has been called the McDuff property and has been of great conceptual importance in the subject.

% In honor of her work, the property of tensorially absorbing the hyperfinite II$_1$ factor is known as the McDuff property (see \cite{McD2}). Ever since, the study of McDuff factors has been of great conceptual importance in the subject. 

The model theory of II$_1$ factors, initiated in  \cite{FarahHartSherman-MTOA1, FarahHartSherman-MTOA2, FarahHartSherman-MTOA3}, has since been an active area of research in operator algebra theory (for a recent survey see \cite{surveyMTOA} and the references therein). Among the various insights that model theory offers is in the analysis of ultrapowers; since the central sequences can be viewed as simply the commutant of the diagonal embedding of a II$_1$ factor $M$ into its ultrapower $M^{\mathcal{U}}$, model theory is especially suited for studying central sequence algebras. (Throughout this introduction, all ultrafilters are assumed to be nonprincipal ultrafilters on $\bb N$). One of the main avenues of research in the model theory of II$_1$ factors is to determine when two II$_1$ factors $M$ and $N$ elementarily equivalent, denoted $M\equiv N$, which, assuming the continuum hypothesis, is equivalent to asking if there exixts an ultrafilter $\mathcal{U}$ such that the ultrapowers $M^\u$ and $N^\u$ are isomorphic.  (In general, in this paper we assume the continuum hypothesis in order to be able to state such algebraic characterizations of model-theoretic properties.) The classification of II$_1$ factors up to elementary equivalence turns out to be much more challenging than the classification up to isomorphism (for instance, there are continuum many non-isomorphic separable II$_1$ factors that are elementarily equivalent to a given II$_1$ factor \cite[Theorem 4.3]{FarahHartSherman-MTOA3}).   The central sequences algebra has again played a crucial role  in the elementary equivalence problem.  Indeed, the first handful of examples of non-elementarily equivalent II$_1$ factors (see \cite{FarahHartSherman-MTOA3, FGL, Dixmier1969Deux, ZM, GoldbringHart-McDuffTheories}) and the first family of continuum many non-elementarily equivalent II$_1$ factors (interestingly coinciding with McDuff's factors; see \cite{BCI15}) pivotally use the structure of the central sequence algebra as the invariant (see also \cite{CIKE}, which exhibits two non-elementarily equivalent II$_1$ factors without central sequences). Despite these results, it remains  a  challenging  problem to further the classification  up to elementary equivalence of II$_1$  factors with central sequences. 
 
In this paper, we study this problem by identifying a proper subfamily of McDuff II$_1$ factors that are stable under elementary equivalence which we call \emph{uniformly super McDuff} factors.  The motivation for the terminology arises from \cite{GoldbringHart-McDuffTheories}, which introduced the super McDuff property, which is the property of admitting a II$_1$ factorial central sequence algebra and which has been of recent interest in the model theory of von Neumann algebras (see for instance \cite{IS, Marrakchi, AGKE-GeneralizedJung, GoldbringHart-FewQfsForThR,   ChifanDI-EmbUnivPptyT}). To see the various examples and relationships  between the super McDuff property and other subclasses of property Gamma factors, we direct the reader to  \cite[Section 6.4]{AGKE-GeneralizedJung}.   This paper introduces a uniform version of a finitary reformulation of the aforementioned super McDuff property, which we call the \textbf{uniform super McDuff property}.

\begin{defn*}
The McDuff II$_1$ factor $M$ is \textbf{uniformly super McDuff} if for all $n\geq 1$, there is $m\geq 1$ such that:  for all $F\subseteq (M)_1$ with $|F|\leq n$, there is $G\subseteq (M)_1$ with $|G|\leq m$ for which, given any $(p,q)\in \cal P_2(M)$ with $$\max_{x\in G}\max(\|[x,p]\|_2,\|[x,q]\|_2)<\frac{1}{m},$$ there is $u\in U(M)$ with $\|upu^*-q\|_2<\frac{1}{n}$ and $\max_{x\in F}\|[x,u]\|_2<\frac{1}{n}$.
\end{defn*}

As a first remark, we recall a notion first introduced in \cite{AGKE-GeneralizedJung} and studied there in connection with proving a generalized version of a key result of Jung \cite{JungTubularity} (see also \cite{ScottSri2019ultraproduct}).

\begin{defn*}
A separable II$_1$ factor $M$ has the \textbf{Brown property} if for every separable subfactor $N\subseteq M^\u$, there is a separable subfactor $N\subseteq P\subseteq M^\u$ with $P'\cap M^\u$ a II$_1$ factor.
\end{defn*}

Recalling that we are assuming the continuum hypothesis, any two nonprincipal ultrapowers of $M$ are isomorphic, implying that the definition of the Brown property is independent of the choice of ultrafilter.  Similarly, one sees that if $M$ has the Brown property and $N$ is elementarily equivalent to $M$, then $N$ also has the Brown property. Our first  main result shows that the uniform  super McDuff property is closed under elementary equivalence and in fact coincides with the Brown property:    

\begin{main} \label{intromain1}
For a II$_1$ factor $M$, the following are equivalent:
\begin{enumerate}
    \item $M$ is uniformly super McDuff.
    \item Every $N\equiv M$ is uniformly super McDuff.
    \item Every $N\equiv M$ is super McDuff.  
    \item $M^\u$ is super McDuff for every nonprincipal ultrafilter $\u$ on $\bb N$.
    \item $M^\u$ is super McDuff for some nonprincipal ultrafilter $\u$ on $\bb N$.
    \item Every separable $N\equiv M$ is super McDuff.
    \item $M$ has the Brown property.
\end{enumerate}
\end{main}

We then proceed to show that two particular families of super McDuff II$_1$ factors are in fact uniformly super McDuff (and hence also have the Brown property):

\begin{main} \label{intromain2}
Any II$_1$ factor that is elementarily equivalent to any of following separable II$_1$ factors is uniformly super McDuff:
\begin{enumerate}
    \item An \textbf{infinitely generic} II$_1$ factor.
    \item $N {\otimes} \R$ where $N$ is a separable full (that is, non-Gamma) II$_1$ factor. 
\end{enumerate}
\end{main}

The first family is a certain subclass of the family of existentially closed factors (see \cite{ecfactors}) arising in the model theory of II$_1$ factors. Our result for this case improves upon the work of \cite{ChifanDI-EmbUnivPptyT}, where it was shown that   infinitely generic II$_1$ factors are super McDuff. Note that our result also obtains uniform super McDuffness for factors elementarily equivalent to  infinitely generic factors, some of which are not infinitely generic themselves (see Lemma \ref{infgenfact}).

The second family of factors in the above list are usually referred to as \textbf{strongly McDuff} factors and were first considered by Popa in \cite{spgap}.  Strongly McDuff II$_1$ factors were shown to be super McDuff in \cite[Proposition 6.1.5]{AGKE-GeneralizedJung}.  

Theorems \ref{intromain1} and \ref{intromain2}  produce several new concrete examples of factors with the Brown property.  Indeed, before this work, the only known examples of II$_1$ factors with the Brown property were those factors elementarily equivalent to $\R$ \cite{AGKE-GeneralizedJung, GoldbringHart-FewQfsForThR}, which do not fall into either of the two categories in the above theorem.

% A natural model-theoretic notion of genericity is provided by the notion of being \textbf{existentially closed} (or e.c. for short). 

Having proven that the infinitely generic factors have the uniform super McDuff property, it is natural to inquire whether another important subclass of the existentially closed factors, namely the finitely generic factors, have this property.  Here, a third property, which we call the (T)-factorial relative commutant property, or (T)-FRC property, becomes relevant.  We say that a separable II$_1$ factor $M$ has the \textbf{(T)-FRC property} if and only if every separable II$_1$ factor with property (T) embeds into $M^\u$ with factorial relative commutant.  We show below that this property is an axiomatizable property of separable II$_1$ factors.  In \cite{Goldbring-PopaFCEP}, it was shown that the infinitely generic factors have the (T)-FRC property.  Here, we connect the question of whether or not finitely generic factors have the (T)-FRC property and whether or not they are uniformly super McDuff: 

\begin{main}
If the finitely generic factors have the (T)-FRC property, then they are also uniformly super McDuff.
\end{main}

 We end this introduction by stating the main open question in connection to our work: 

 \textbf{Question D.}
Is the class of II$_1$ factors with the super McDuff property closed under elementary equivalence?  Equivalently, do the classes of super McDuff and uniformly super McDuff factors coincide?

\subsection*{Organization of the paper}

In Section 2, we recall some background on the model theory of tracial von Neumann algebras, including a discussion of the classes of existentially closed, infinitely generic, and finitely generic factors. For a more detailed background on related matters, we direct the reader to \cite[Section 2]{AGKE-GeneralizedJung}. Section 3 is where the  uniform super McDuff property is introduced and Theorem \ref{intromain1} is proved. Section 4 proves Theorem \ref{intromain2} via a unified approach involving a certain spectral gap criterion. In Section 5, Theorem C is proved and various remarks surrounding the problem of whether finitely generic II$_1$ factors are uniformly super McDuff are documented.

\subsection*{Acknowledgements} It is our pleasure to thank Adrian Ioana for suggesting to us the idea of using Hastings' result for proving Theorem \ref{intromain2} in Section \ref{subsec:stronglyMcDuff}.
\section{Preliminaries}

\subsection{Existentially closed, infinitely generic, and finitely generic factors}\label{subsection-e.c.}

\begin{defn}
A separable tracial von Neumann algebra $M$ is said to be \textbf{existentially closed (e.c.)} if: for any tracial von Neumann algebra $N$ containing $M$, there is an embedding $N\hookrightarrow M^\u$ that restricts to the diagonal embedding $M\hookrightarrow M^\u$.
\end{defn}

The definition given above is the semantic definition and uses the separability assumption on $M$.  The syntactic definition, which makes no separability assumption, reads:  $M$ is e.c. if and only if, for every tracial von Neumann algebra $N$ extending $M$ and every quantifier-free formula $\varphi(x)$ with parameters from $M$, we have 
$$\inf\{\varphi(a)^M \ : \ a\in M\}=\inf\{\varphi(b)^N \ : \ b\in N\}.$$

E.c. tracial von Neumann algebras are McDuff II$_1$ factors (see, for example, \cite[Section 2]{GoldbringHartSinclair-NoModelCompanion}); for this reason, we often refer to e.c.\ tracial von Neumann algebras as e.c.\ II$_1$ factors.  Moreover, the class of e.c. tracial von Neumann algebras is \textbf{embedding universal}, meaning that every separable II$_1$ factor is contained in a separable e.c. II$_1$ factor.  (This is a special case of a general model-theoretic fact; see \cite[Lemma 3.5.7]{CK2013ModelTheory} for a proof in classical logic.)

Before moving on, we make a short digression to point out some facts about e.c. factors that follow from the results of \cite{AGKE-GeneralizedJung} but were not explicitly pointed out there.  In \cite[Question 6.3.1]{AGKE-GeneralizedJung}, it was asked whether or not every e.c. factor is super McDuff.  The following yields a criterion for when an e.c. factor is super McDuff:

\begin{thm}\label{ecsupercriterion}
For an e.c. factor $M$, the following are equivalent:
\begin{enumerate}
    \item $M$ is super McDuff.
    \item For any e.c. factor $N$ with $M\subseteq N$, we have that $M'\cap N^\u$ is a factor.
    \item There is an e.c. factor $N$ with $M\subseteq N$ for which $M'\cap N^\u$ is a factor.
\end{enumerate}
\end{thm}

\begin{proof}
For the implication (1) implies (2), suppose that $M$ is super McDuff and $N$ is an e.c. factor with $M\subseteq N$.  By considering the chain $M\subseteq M^\u\subseteq N^\u$, we see that $M'\cap N^\u$ is a factor using \cite[Theorem 7.15(2), Lemma 7.2.1, and Lemma 7.2.8]{AGKE-GeneralizedJung}.

The implication (2) implies (3) is trivial.  For the implication (3) implies (1), suppose that $N$ is e.c., $M\subseteq N$, and $M'\cap N^\u$ is a factor.  We then have that $M\subseteq M^\u$ is a factor using \cite[Theorem 7.1.5(1) and Lemma 7.2.1]{AGKE-GeneralizedJung}.
\end{proof}

In \cite[Question 6.3.10]{AGKE-GeneralizedJung}, it was asked if a union of a chain of super McDuff factors is once again super McDuff.  Using Theorem \ref{ecsupercriterion}, we can show that the answer to this question is positive if all of the super McDuff factors in the chain are also e.c. (which is a strenthening of \cite[Proposition 6.3.17]{AGKE-GeneralizedJung}): 

\begin{cor}
Any union of a chain of e.c. super McDuff factors is (e.c.) super McDuff.
\end{cor}

We now return to the main thread of this section.  In this paper, we will  be concerned with specific subclasses of the class of e.c. factors.  We first consider the class of \textbf{infinitely generic factors}, which is defined to be the largest class of II$_1$ factors that is embedding universal and which is \textbf{model-complete}, meaning any embedding $i:M\hookrightarrow N$ between two infinitely generic factors is elementary, that is, extends to an isomorphism $M^\u\to N^\u$.  (The existence of such a class follows from \cite[Corollary 5.15]{ecfactors}.  The terminology infinitely generic refers to the use of so-called infinite model-theoretic forcing in the construction of these factors; see \cite[Section 5]{ecfactors}.)  It follows from \cite[Lemma 5.20]{ecfactors} that any two infinitely generic II$_1$ factors are elementarily equivalent.  That infinitely generic factors are e.c. follows from \cite[Proposition 5.11]{ecfactors}.

In connection with Theorem \ref{thm:mainthm2} below, we need to record the following:

\begin{lem}\label{infgenfact}
There is a II$_1$ factor $M$ that is elementarily equivalent to the infinitely generic factors but is not itself infinitely generic.
\end{lem}

\begin{proof}
Let $T^*$ denote the common theory of the infinitely generic factors.  If the statement of the lemma is false, then the models of $T^*$ are precisely the infinitely generic factors themselves.  Since the collection of infinitely generic factors is embedding universal and model-complete, it follows that $T^*$ is the model companion of the theory of tracial von Neumann algebras, contradicting the main result of \cite{GoldbringHartSinclair-NoModelCompanion}.
\end{proof}

The other class of e.c. factors we will need to consider are the \textbf{finitely generic factors}.  Any definition of these factors involves a discussion of finite model-theoretic forcing (see, for example, \cite[Section 6]{ecfactors} or \cite[Section 3]{Goldbring-Enforceable}).  However, for our purposes, we will only need to know the following two properties of a finitely generic factor $N$ (besides their existence):

\begin{itemize}
    \item $N$ is \textbf{locally universal}, meaning that every separable II$_1$ factor admits an embedding into $N^\u$;
    \item $N$ has the \textbf{generalized Jung property}, meaning that every embedding $N\hookrightarrow N^\u$ is elementary.
\end{itemize}

The first item follows from \cite[Propositions 2.10 and 3.9]{Goldbring-Enforceable} and the fact that e.c. models of theories with the joint embedding property are locally universal (see \cite[Subsection 1.1]{Goldbring-Enforceable}).  The second item follows from \cite[Proposition 3.10]{Goldbring-Enforceable}.

Whether these two properties actually characterize being finitely generic is an open problem (see \cite[Question 3.3.8.]{AGKE-GeneralizedJung}).  

It follows immediately from these two properties that any finitely generic factor is e.c.  It is unknown whether or not the theories of finitely generic and infinitely generic II$_1$ factors coincide, or equivalently, whether or not any finitely generic factor must be infinitely generic.  (In general, there need not be any connection between the finitely generic and infinitely generic models of a given theory.)

\subsection{Definability}\label{subsection:definability}

We freely use the terminology around definability adopted in \cite{Goldbring-gap}.  Throughout this subsection, $T$ denotes the theory of tracial von Neumann algebras in the appropriate continuous signature.

For a tracial von Neumann algebra $M$, we let $\cal P_2(M)$ denote the set of pairs of projections from $M$ of the same trace.  The relevance of the set $\cal P_2(M)$ is that $M$ is a factor if and only if, for any $(p,q)\in \cal P_2(M)$, there is a unitary $u\in M$ with $upu^*=q$. 

The assignment $M\mapsto \cal P_2(M)$ is a $T$-functor which is the zeroset of the $T$-formula $$\xi(p,q):=\max(\|p^*-p\|_2,\|p^2-p\|_2,\|q^*-q\|_2,\|q^2-q\|_2,|\tau(p)-\tau(q)|).$$ Moreover, an easy functional calculus argument shows that $$\cal P_2\left(\prod_\u M_i \right)=\prod_\u \cal P_2(M_i)$$ for all families $(M_i)_{i\in I}$ of tracial von Neumann algebras and all ultrafilters $\u$ on $I$, whence $\cal P_2$ is a $T$-definable set by \cite[Theorem 2.13]{Goldbring-gap} .  

Consequently, for any $T$-formula $\varphi(p,q,\vec z)$, the $T$-functor $\sup_{(p,q)\in \cal P_2}\varphi(p,q,\vec z)$ is a $T$-formula.  It will behoove us to understand the relationship between the quantifier-complexity of $\sup_{(p,q)\in \cal P_2}\varphi(p,q,\vec z)$ and of $\varphi(p,q,\vec z)$ itself.  To understand this relationship, we first note that, as $T$-formulae, one has $$\sup_{(p,q)\in \cal P_2}\varphi(p,q,\vec z)=\sup_{\vec x}(\varphi(\vec x,\vec z)\dminus \alpha( d(\vec x,\cal P_2)),$$ where $\alpha:[0,1]\to [0,1]$ is an increasing continuous function with $\alpha(0)=0$ satisfying $$|\varphi(\vec x,\vec z)-\varphi(\vec y,\vec z)|\leq \alpha(d(\vec x,\vec y))$$ for all tuples $\vec x$, $\vec y$, $\vec z$ in models of $T$.  Since $\xi$ is quantifier-free, one can show that $d(\vec x,\cal P_2)$ is $T$-equivalent to an existential formula (see \cite[Section 2]{GoldbringHart-undecidability}, whence so is $\alpha(d(\vec x,\mathcal{P}_2))$ as $\alpha$ is increasing.  As a result, if $\varphi(p,q,\vec z)$ is $\forall_n$ (resp. $\exists_n$), then the $\sup_{(p,q)\in \cal P_2}\varphi(p,q,\vec z)$ is $\forall_n$ (resp. $\forall_{n+1}$).

We will also find the need to take the infimum over the $T$-definable set $U$ of unitaries, which is the zeroset of the $T$-formula $$\upsilon(u)=\max(\|u^*u-1\|_2,\|uu^*-1\|_2).$$  In this case, for any $T$-formula $\varphi(u,\vec z)$, there is a continuous function $\gamma$ such that, as $T$-formulas, we have $$\inf_{u\in U}\varphi(u,\vec z)=\inf_x(\varphi(x,\vec z)+\gamma(\upsilon(x))$$ (see the proof of \cite[Theorem 2.13]{Goldbring-gap}).  As a result, if $\varphi(u,\vec z)$ is $\forall_n$ (resp. $\exists_n$), then $\inf_{u\in U}\varphi(u,\vec z)$ is $\exists_{n+1}$ (resp. $\exists_{n}$).

\section{The uniform super McDuff property}

The following is a finitary (ultraproduct-free) reformulation of having factorial relative commutant:

\begin{prop}\label{finitary}
If $M$ is a separable II$_1$ factor, then $M'\cap M^\u$ is a factor if and only if:  for every $\epsilon>0$ and every finite $F\subseteq (M)_1$, there are $\delta>0$ and finite $G\subseteq (M)_1$ such that: for all $(p,q)\in \cal P_2(M)$, if $\|[x,p]\|_2,\|[x,q]\|_2<\delta$ for all $x\in G$, then there is $u\in U(M)$ such that $\|upu^*-q\|_2<\epsilon$ and $\|[x,u]\|_2<\epsilon$ for all $x\in F$.
\end{prop}

\begin{proof}
First suppose that $M'\cap M^\u$ is a factor and yet, towards a contradiction, the condition fails for a particular $\epsilon$ and $F$.  Let $(G_n)$ denote a sequence of finite subsets of $(M)_1$ whose union is SOT-dense in $(M)_1$.  Then by the contradiction assumption, for each $n\geq 1$, there are projections $(p_n,q_n)\in \cal P_2(M)$ witnessing that $\delta=\frac{1}{n}$ and $G=G_n$ do not satisfy the conclusion of the forward direction for this particular choice of $\epsilon$ and $F$.  Set $p=(p_n)_\u$ and $q=(q_n)_\u$.  Note then that $(p,q)\in \cal P_2(M'\cap M^\u)$.  Since $M$ is super McDuff, there is $u\in U(M'\cap M^\u)$ such that $upu^*=q$.  Consequently, for $\u$-almost all $n$, we have that $\|u_np_nu_n^*-q_n\|_2<\epsilon$ and $\max_{x\in F}\|[x,u_n]\|_2<\epsilon$, a contradiction to the choice of $(p_n,q_n)$.

For the converse direction, suppose the condition holds.  Fix $(p,q)\in \cal P_2(M'\cap M^\u)$; we aim to show that $p$ and $q$ are unitarily conjugate in $M'\cap M^\u$.  By saturation, it suffices to show that, for each $\epsilon>0$ and finite $F\subseteq (M)_1$, there is $u\in U(M^\u)$ such that $\|upu^*-q\|_2\leq \epsilon$ and $\max_{x\in F}\|[x,u]\|_2\leq \epsilon$.  Let $\delta$ and $G$ be as in the condition for $\epsilon$ and $F$.  Write $p=(p_n)_\u$ and $q=(q_n)_\u$ with $(p_n,q_n)\in \cal P_2(M)$ for all $n\in \bb N$.  Note that, for $\u$-almost all $n$, we have $\max_{x\in G}\max(\|[x,p_n]\|_2,\|[x,q_n]\|_2)<\delta$.  Consequently, for these $n$, there is $u_n\in U(M)$ such that $\|u_np_nu_n^*-q_n\|_2<\epsilon$ and $\max_{x\in F}\|[x,u_n]\|_2<\epsilon$.  Let $u=(u_n)_\u\in U(M^\u)$.  Then $\|upu^*-q\|_2\leq \epsilon$ and $\max_{x\in F}\|[x,u]\|_2\leq \epsilon$, as desired.
\end{proof}

\begin{remark}
While the original definition of super McDuff was only for separable II$_1$ factors, we can use the finitary criterion appearing in Proposition \ref{finitary} to give a definition of super McDuffness for arbitrary II$_1$ factors (using also that there is a finitary definition of McDuffness that makes sense for arbitrary tracial von Neumann algebras; see \cite[Proposition 3.9]{FarahHartSherman-MTOA3}).
\end{remark}

The condition in Proposition \ref{finitary} lends itself naturally to a uniform version, namely that, in the criterion of Proposition \ref{finitary}, both $\delta$ and the size of $G$ may be taken to depend only on $\epsilon$ and the size of $F$.  Said somewhat differently:

\begin{defn}
The McDuff II$_1$ factor $M$ is \textbf{uniformly super McDuff} if and only if, for all $n\geq 1$, there is $m\geq 1$ such that:  for all $F\subseteq (M)_1$ with $|F|\leq n$, there is $G\subseteq (M)_1$ with $|G|\leq m$ for which, given any $(p,q)\in \cal P_2(M)$ with $$\max_{x\in G}\max(\|[x,p]\|_2,\|[x,q]\|_2)<\frac{1}{m},$$ there is $u\in U(M)$ with $\|upu^*-q\|_2<\frac{1}{n}$ and $\max_{x\in F}\|[x,u]\|_2<\frac{1}{n}$.
\end{defn}

\begin{remark}\label{nonGammauniform}
    Note that the uniformity condition described in the previous definition is automatic in the case that $M$ does not have property Gamma.  Indeed, in this situation, given $n\geq 1$, there is a finite subset $K\subseteq (M)_1$ and $m\geq 1$ such that, for all contractions $x\in M$, if $\max_{y\in K}\|[x,y]\|_2<\frac{1}{m}$, then $d(x,\bb C)<\frac{1}{2n}$.  Consequently, if $(p,q)\in \cal P_2(M)$ are such that $\max_{y\in K}\max(\|[p,y]\|_2,\|[q,y]\|_2)<\frac{1}{m}$, then $d(p,\bb C),d(q,\bb C)<\frac{1}{2n}$, whence $\|p-q\|_2<\frac{1}{n}$ and we may take $u=1$ in the above definition.
\end{remark}

Now we are ready to prove that the Brown property is equivalent to the uniform super McDuff property, as well as several other conditions:

\begin{thm} \label{prop: list-equiv-ppties}
For a II$_1$ factor $M$, the following are equivalent:
\begin{enumerate}
    \item $M$ is uniformly super McDuff.
    \item Every $N\equiv M$ is uniformly super McDuff.
    \item Every $N\equiv M$ is super McDuff.  
    \item $M^\u$ is super McDuff for every nonprincipal ultrafilter $\u$ on $\bb N$.
    \item $M^\u$ is super McDuff for some nonprincipal ultrafilter $\u$ on $\bb N$.
    \item Every separable $N\equiv M$ is super McDuff.
    \item $M$ has the Brown property.
\end{enumerate}
\end{thm}

\begin{proof}
(1) implies (2): Suppose that $M$ is uniformly super McDuff.  Fix $n\geq 1$ and take the corresponding $m\geq 1$ as in the definition of uniformly super McDuff.  Consider the sentence $\sigma_{n,m}=\sup_{x_1,\ldots,x_n}\inf_{y_1\ldots,y_m}\sup_{(p,q)\in \cal P_2}\varphi_{n,m}(\vec x,\vec y,p,q)$, where $\varphi_{n,m}(\vec x,\vec y,p,q)$ is the following formula:
\begin{multline*}
\min\biggl(\frac{1}{m} \dminus \max_{j=1,\ldots,m}\max(\|[p,y_j]\|_2,\|[q,y_j]\|_2), 
\inf_{u\in U(M)}\max(\|upu^*-q\|_2,\max_{i=1,\ldots,n}\|[[u,x_i]\|_2)\dminus \frac{1}{n} \biggr).
\end{multline*}
Then $\sigma_{n,m}^M=0$ by assumption, whence the same is true for any elementarily equivalent structure $N$.  We use this to show that $N$ is uniformly super McDuff.  Fix $\eta>0$ with $\eta<\min(\frac{1}{n},\frac{1}{m})$; we claim that $\delta:=\frac{1}{m}-\eta$ and $m$ witnesses the uniform super McDuffness of $N$ for finite sets of size at most $n$ and with $\epsilon=\frac{2}{n}$.  To see this, fix $x_1,\ldots,x_n\in N$ and take $y_1,\ldots,y_m\in N$ so that $\sup_{(p,q)\in \cal P_2}\varphi_{n,m}(\vec x,\vec y,p,q)<\eta$.  Now suppose that $(p,q)\in \cal P_2(N)$ is such that $$\max_{j=1,\ldots,m}\max(\|[p,y_j]\|_2,\|[q,y_j]\|_2)<\frac{1}{m}-\eta.$$  It follows that the first entry of $\varphi_{n,m}(\vec x,\vec y,p,q)$ is greater than $\eta$, whence the second entry is smaller than $\eta$.  Consequently, there is a unitary which conjugates $p$ to within $\frac{1}{n}+\eta < \frac{2}{n}$ of $q$ and which $\frac{1}{n}+\eta < \frac{2}{n}$ commutes with each $x_i$, as desired.

The implications (2) implies (3),  (3) implies (4), and (4) implies (5) are all trivial.  To see that (5) implies (6), assume that $M^\u$ is super McDuff and fix a separable $N\equiv M$.  By considering an elementary embedding $N\hookrightarrow M^\u$, we may assume that $N$ is an elementary substructure of $M^\u$.  Fix $n\geq 1$ and finitely many $x_1,\ldots,x_n\in N$.  Since $M^\u$ is super McDuff, we get a corresponding $\delta>0$ and finite $G\subseteq (M^\u)_1$.  If $m\geq \frac{1}{\delta}$ and $G=\{y_1,\ldots,y_m\}$ (possibly with repetitions), we have that $(\sup_{(p,q)\in \cal P_2}\varphi_{n,m}(\vec x,\vec y,p,q))^{M^\u}=0$, where $\varphi_{n,m}$ is as in the implication (1) implies (2).  Since $N$ is an elementary substructure of $M^\u$, it follows that $$\left(\inf_{\vec y}\sup_{(p,q)\in \cal P_2}\varphi_{n,m}(\vec x,\vec y,p,q)\right)^N=0.$$  If $z_1,\ldots,z_m\in N$ witness that the existential is less than $\eta<\min(\frac{1}{n},\frac{1}{m})$, as in the implication (1) implies (2), we see that $\delta=\frac{1}{m}$ and $z_1,\ldots,z_m$ witness the super McDuffness of $N$ for $x_1,\ldots,x_n$ and the error $\frac{2}{n}$.

(6) implies (4):  Fix $n\geq 1$ and $x_1,\ldots,x_n\in M^\u$.  Take a separable elementary substructure $N$ of $M^\u$ containing $x_1,\ldots,x_n$.  By (6), $N$ is super McDuff.  Let $m\geq 1$ be such that $\delta=\frac{1}{m}$ and $y_1,\ldots,y_m\in N$ are as in the definition of super McDuffness for $\epsilon=\frac{1}{n}$ and $x_1,\ldots,x_n$.  It follows that $(\sup_{(p,q)\in \cal P_2}\varphi_{n,m}(\vec x,\vec y,p,q))^N=0$.    By elementarity, we have that $(\sup_{(p,q)\in \cal P_2}\varphi_{n,m}(\vec x,\vec y,p,q))^{M^\u}=0$.  Thus, $\frac{1}{m}$ and $y_1,\ldots,y_m$ witness super McDuffness of $M^\u$ for $\epsilon=\frac{1}{n}$ and $x_1,\ldots,x_n$.

(4) implies (1): Suppose, towards a contradiction, that $M$ is not uniformly super McDuff.  Then there is some $n\geq 1$ such that no pair $m$ witnesses uniform super McDuffness of $M$ for this $n$.  Consequently, for each $m\geq 1$, there are $x_1^m,\ldots,x_n^m\in M$ such that no $G\subseteq (M)_1$ of size at most $m$  witnesses the truth of super McDuffness for $M$.  We can then consider the elements $x_j=(x_j^m)_\u\in M^\u$.  By assumption, $M^\u$ is super McDuff, so there are $\delta>0$ and finite $G=\{y_1,\ldots,y_t\}\in M^\u$ that witness super McDuffness for the finite subset $\{x_1,\ldots,x_n\}$ of $M^\u$ and $\epsilon=\frac{1}{n}$.  Write $y_j=(y_j^m)_\u$ for each $j=1,\ldots,t$.  By assumption, for each $m\geq t$, there is $(p_m,q_m)\in \cal P_2(M)$ with $\max_{j=1,\ldots,t}(\|[p_m,y_j^m]\|_2,\|[q_m,y_j^m]\|_2)<\frac{1}{m}$ and yet no unitary in $M$ conjugates $p_m$ to within $\frac{1}{n}$ of $q_m$ while also $\frac{1}{n}$-commuting with each $x_j^m$ for $j=1,\ldots,n$.  Let $p=(p_m)_\u$ and $q=(q_m)_\u$, whence $(p,q)\in \cal P_2(M^\u)$ and $[p,y_j]=[q,y_j]=0$ for all $j=1,\ldots,t$.  It follows that there is $u\in U(M^\u)$ such that $\|upu^*-q\|_2<\frac{1}{n}$ and $\max_{j=1,\ldots,n}\|[u,x_j]\|_2< \frac{1}{n}$.  Write $u=(u_m)_\u$ with each $u_m\in U(M)$.  Then for $\u$-almost all $m$, we have $\|u_mp_mu_m^*-q_m\|_2<\frac{1}{n}$ and $\max_{j=1,\ldots,n}\|[u_m,x_j^m]\|_2<\frac{1}{n}$, contradicting the choice of $p_m$ and $q_m$.

At this point, we have that (1)-(6) are equivalent.  However, as mentioned in the introduction, it was shown in \cite[Proposition 6.2.4]{AGKE-GeneralizedJung} that (6) and (7) are equivalent, finishing the proof of the theorem.
\end{proof}

\begin{remark}
By the comments made at the end of Subsection \ref{subsection:definability}, we see that the sentences appearing in the above proof are $\forall_4$-sentences.  The quantifier-complexity of these concrete sentences agrees with the fact established in \cite{GoldbringHart-FewQfsForThR}, proven using Ehrenfeucht-Fra\"isse games, that any II$_1$ factor with the same four quantifier theory as a II$_1$ factor with the Brown property also has the Brown property.
\end{remark}

Since $\R$ has the Brown property, by Theorem \ref{prop: list-equiv-ppties}, we know that $\R$ is uniformly super McDuff.  In Subsection \ref{subsec:stronglyMcDuff} below, we will give a direct argument that $\R$ is uniformly super McDuff.

\section{Spectral gap and the uniform super McDuff property}

In this section, we describe a sufficient condition for the uniform super McDuff property in terms of spectral gap, and deduce that infinitely generic II$_1$ factors and strongly McDuff II$_1$ factors are both uniformly super McDuff.

\subsection{Spectral gap criterion}

We begin by recalling the definition of spectral gap for a subalgebra of a tracial factor.

\begin{defn}
Let $M$ be a tracial factor and $A \subseteq M$ a von Neumann subalgebra. Then $A \subseteq M$ has \textbf{spectral gap} if there exists a finite $F \subseteq (A)_1$ and a constant $C>0$ such that, for all $\xi \in L^2(M)$, we have
\[
\norm{\xi - E_{A' \cap M}(\xi)}_2^2 \leq C \sum_{x \in F} \norm{x \xi - \xi x}_2^2.
\]
If, in addition, $|F|\leq C$, then we call $C$ a \textbf{spectral gap number} for $A\subseteq M$.  We let $\operatorname{SG}(A,M)$ denote the smallest spectral gap number for $A$ inside $M$.  We also write $\operatorname{SG}(M) = \operatorname{SG}(M,M)$.
\end{defn}

\begin{remark}
    In particular, $M \subseteq M$ has spectral gap if there exists a finite $F\subseteq (M)_1$ and a constant $C>0$ such that, for all $\xi\in L^2(M)$, we have
\[
\norm{\xi - \tau(\xi)\cdot 1}_2^2 \leq C \sum_{x \in F} \norm{x \xi - \xi x}_2^2.
\]

Connes showed in \cite[Theorem 2.1]{Connes1976} that a tracial factor $M\subseteq M$ has spectral gap if and only if $M$ is full (equivalently $M$ does not have property Gamma in the case of a II$_1$ factor).
\end{remark}

% \begin{defn}
% We call $C$ a \textbf{spectral gap number} for $A \subseteq M$ if there exists a finite $F$ inside the unit ball of $A$ such that $|F| \leq C$ and
% \[
% \norm{\xi - E_{A' \cap M}(\xi)}_2^2 \leq C \sum_{x \in F} \norm{x \xi - \xi x}_2^2 \text{ for } \xi \in L^2(M).
% \]

% \end{defn}

For a tracial von Neumann algebra $M$, subsets $X,Y\subseteq M$, and $\epsilon>0$, we write $X\subseteq_\epsilon Y$ to mean that every element of $X$ is within $\epsilon$ (with respect to the 2-norm) of some element of $Y$.

\begin{defn}
We say that a tracial factor $M$ has the \textbf{factorial commutant limit spectral gap property} or \textbf{FC limit spectral gap property} if there exists some constant $s>0$ such that, for any finite $F \subseteq M$ and $\epsilon > 0$, there is some von Neumann subalgebra $A \subseteq M$ such that $A' \cap M$ is a factor, $\operatorname{SG}(A,M) \leq s$, and $F \subseteq_\epsilon A$.
\end{defn}

The following proposition provides a sufficient criterion for a McDuff II$_1$ factor to be uniformly super McDuff:

\begin{prop}\label{sgapcriterion}
If $M$ is a McDuff II$_1$ factor with the FC limit spectral gap property, then $M$ is uniformly super McDuff.
\end{prop}

\begin{proof}
Fix $n \geq 1$.  We claim that any $m > 24sn$ witnesses that $M$ is uniformly super McDuff.  Towards that end, fix $F \subseteq (M)_1$ with $|F| \leq n$.  Choose $A \subseteq M$ such that $A' \cap M$ is a factor, $\operatorname{SG}(A,M) \leq s$, and $F$ is contained in the $(1/2n)$-neighborhood of $A$.  Then there exists $G \subseteq (A)_1$ such that $|G| \leq s$ and
\[
\norm{\xi - E_{A'\cap M}(\xi)}_2^2 \leq s \sum_{x \in G} \norm{x \xi - \xi x}_2^2.
\]
In particular, if $(p,q) \in \mathcal{P}_2(M)$ and $\max_{x \in G} \max(\norm{[x,p]}_2,\norm{[x,q]}_2) \leq 1/m$, then
\[
\norm{p - E_{A' \cap M}(p)}_2^2 \leq \frac{s^2}{m^2}.
\]
Set $a = E_{A' \cap M}(p)$.  Since $\norm{a^2 - p^2}_2 < 2s/m$, we get $\norm{a - a^2}_2 < 3s/m$.  Set $\tilde{p} = \mathbf{1}_{[1/2,1]}(a)$, a projection in $A' \cap M$.  Since $|t - \mathbf{1}_{[1/2,1]}(t)| \leq 2t(1 - t)$ for $t \in [0,1]$, we have $\norm{p}_2 \leq 2 \norm{a(1 - a)}_2 < 6s/m$.  Similarly, there is a projection $\tilde{q}\in A'\cap M$ such that $\|\tilde{q}-q\|_2<6s/m$.

Next notice that $|\tau(\tilde{p}) - \tau(\tilde{q})| \leq 12s/m$.  Since $A' \cap M$ is a factor, there exists a unitary $u$ in $A' \cap M$ such that $u\tilde{p}u^* - \tilde{q}$ is plus or minus a projection of trace $< 12s/m$.  In particular, $\norm{u \tilde{p} u^* - \tilde{q}}_2 < 12s/m$, which implies $\norm{upu^* - q}_2 < 24s/m<1/n$.

Moreover, since we assumed that $F$ is contained in the $(1/2n)$-neighborhood of $A$ and $u$ belongs to $A' \cap M$, we have that $\norm{[u,x]}_2 < 1/n$ for $x \in F$.  Finally, note that $|G| \leq s \leq 24sn < m$, finishing the proof.
\end{proof}

\begin{remark}
It is interesting to note that in the previous proof, the size of $G$ only depends on the constant $s$ and does not depend on $n$.
\end{remark}

% \begin{remark}
% One could also replace spectral gap with weak spectral gap.  Although weak spectral gap cannot be quantified by a single number, we can still consider a II$_1$ factor $M$ such that any finite subset $F$ is $\epsilon$-contained in some subalgebra $A$ with weak spectral gap that is witnessed in a uniform manner.  However, the strong spectral gap case is sufficient for the applications in this paper.
% \end{remark}

\subsection{Infinitely generic factors are uniformly super McDuff}

In this section, we use Proposition \ref{sgapcriterion} to show that infinitely generic factors are uniformly super McDuff.  This result improves the result of Chifan, Drimbe, and Ioana from \cite{ChifanDI-EmbUnivPptyT} showing that infinitely generic factors are super McDuff. A main ingredient in the proof of the latter result is the embedding universality of property (T) factors, also proved in \cite{ChifanDI-EmbUnivPptyT}.  Our proof in this subsection also utilizes this fact (and its proof).

% We use a similar technique to Chifan, Drimbe, and Ioana's \cite{ChifanDI-EmbUnivPpty(T)} result on infinitely generic factors, based on subfactors with property (T).

Recall that a II$_1$ factor $N$ has \textbf{property (T)} if there is a finite $F\subseteq N$ and $\delta>0$ such that, whenever $\mathcal{H}$ is a normal $N$-$N$ bimodule with a unit vector $v\in \mathcal{H}$ satisfying $\max_{x\in F}\|x \xi - \xi x\|<\delta$, then there is a nonzero vector $\eta \in \mathcal{H}$ such that $x\eta = \eta x$ for all $x\in N$.  If $\Gamma$ is an ICC group, then $L(\Gamma)$ has property (T) if and only if $\Gamma$ has property (T).  Connes and Jones \cite{propertyT} also showed that property (T) for factors is equivalent to the following, seemingly stronger, property.  

\begin{thm}
    A II$_1$ factor has property (T) if and only if there is a finite subset $F \subseteq N$ (called a \textbf{Kazhdan set}) and a constant $K > 0$ such that, for every normal $N$-$N$ bimodule $\mathcal{H}$, we have
\[
\norm{\xi - P_{\text{central}}(\xi)}_{\mathcal{H}}^2 \leq K \sum_{x \in F} \norm{x \xi - \xi x}_{\mathcal{H}}^2,
\]
where $P_{\text{central}}$ is the projection onto the subspace of central vectors in $\mathcal{H}$.
\end{thm}

In particular, if $N \subseteq M$ is an inclusion of II$_1$ factors and $N$ has property (T), then $L^2(M)$ is a normal $N$-$N$-bimodule and therefore $N$ has spectral gap inside $M$.  Moreover, Tan \cite{Tan2022} showed the converse implication, namely that if a II$_1$ factor $N$ has spectral gap in $M$ for every inclusion $N \subseteq M$, then $N$ has property (T).  Hence, it is natural to leverage property (T) subalgebras of $M$ together with our spectral gap criterion for the uniform super McDuff property.

Parallel to the spectral gap number in the previous section, we introduce the following notion.

\begin{defn}
We say that $K$ is a \textbf{Kazhdan number} for a II$_1$ factor $N$ with property (T) if there exists a finite $F \subseteq (N)_1$ with $|F| \leq K$ such that for every normal $N$-$N$-bimodule $\mathcal{H}$,
\[
\norm{\xi - P_{\text{central}}(\xi)}_H^2 \leq K \sum_{x \in F} \norm{x \xi - \xi x}_H^2.
\]
We denote by $\operatorname{K}(N)$ the smallest possible Kazhdan number for $N$.
\end{defn}

Observe that if $N$ is a II$_1$ factor with property (T) and $N \subseteq M$ is an inclusion of tracial von Neumann algebras, then $\operatorname{SG}(N,M) \leq \operatorname{K}(N)$.

\begin{defn}
We say that a  II$_1$ factor $M$ has the \textbf{uniform local property (T)} if there exists some constant $s$ such that, for any finite $F \subseteq M$ and $\epsilon > 0$, there is some II$_1$ factor $N \subseteq M$ with property (T) such that $\operatorname{K}(N) \leq s$ and $F \subseteq_\epsilon A$.
\end{defn}

\begin{prop} \label{uniformpropT}
Let $M$ be an e.c.\ II$_1$ factor with the uniform local property (T).  Then $M$ has the FC limit spectral gap property, and hence is uniformly super McDuff.
\end{prop}

\begin{proof}
Let $s$ be as in the definition of uniform local property (T).  If $F \subseteq M$ is a finite subset, then there exists some II$_1$ factor $N \subseteq M$ with property (T) such that $\operatorname{K}(N) \leq s$ and $F \subseteq_\epsilon N$.  Note $\operatorname{SG}(N,M) \leq \operatorname{K}(N) \leq s$.  Moreover, since $M$ is e.c.\ and $N$ has property (T), it follows that $N' \cap M$ is a factor (see \cite[Theorem 2.8]{Goldbring-PopaFCEP}).  Hence, $M$ has the FC limit spectral gap property.  Moreover, since e.c.\ II$_1$ factors are McDuff, Proposition \ref{sgapcriterion} implies that $M$ is uniformly super McDuff.
\end{proof}

\begin{remark}
One can make a similar argument if we drop the uniformity in the definition of uniform local property (T). Say that the II$_1$ factor $M$ has the \textbf{local property (T)} if:  for every $\epsilon>0$ and finite $F\subseteq M$, there is a property (T) subfactor $N$ of $M$ such that $F\subseteq_\epsilon N$.  Then an e.c.\ factor $M$ with local property (T) will be super McDuff.  Indeed, $M$ can be written as an inductive limit of property (T) subfactors $N_i$; just as in the previous proof, since $M$ is e.c., each $N_i' \cap M$ is a factor.  This implies that $N_i' \cap M^{\u}$ is a factor, and therefore $M'\cap M^\u$ is a factor by \cite[Proposition 6.3.12]{AGKE-GeneralizedJung}.
\end{remark}

%OLD VERSION

%\begin{defn}
% We say that the II$_1$ factor $M$ has the \textbf{local property (T)} if:  for every $\epsilon>0$ and finite $F\subseteq M$, there is a property (T) subfactor $N$ of $M$ such that $F\subseteq_\epsilon N$.
% \end{defn}

% \begin{prop}\label{localTprop}
% Suppose that $M$ is a separable e.c. II$_1$ factor with the local property (T). 
% Then $M$ is super McDuff.
% \end{prop}

% \begin{proof}
% By assumption, there is a sequence $(N_n)$ of property (T) subfactors of $M$ with SOT-dense union.  Since $M$ is e.c., we have that $N_n'\cap M$ is a factor (see \cite[Theorem 2.8]{Goldbring-PopaFCEP}), whence so is $N_n'\cap M^\u$.  It follows that $M'\cap M^\u$ is also a factor by \cite[Proposition 6.3.12]{AGKE-GeneralizedJung}.  
% \end{proof}

Next, we will show that the previous proposition applies to \emph{some} infinitely generic II$_1$ factors.

\begin{prop} \label{prop:infgenericprop(T)}
There exists an infinitely generic II$_1$ factor with the uniform local property (T).
\end{prop}

To prove this, we will use the following Theorem, which is extracted from the discussion at the end of the introduction to \cite{ChifanDI-EmbUnivPptyT}.

\begin{thm}[\cite{ChifanDI-EmbUnivPptyT}]\label{propTfact}
There is a fixed countable property (T) group $\Gamma$ such that:  given any II$_1$ factor $M$, there is a II$_1$ factor $N$ containing $M$ that is generated by a group homomorphism $\pi:\Gamma\to U(N)$.  Consequently, $N$ has property (T) and a Kazhdan set for $N$ can be taken to be the image under $\pi$ of a Kazhdan set for $\Gamma$ with associated Kazhdan constant equal to that of $\Gamma$.  As a result, the Kazhdan numbers of $N$ are uniformly bounded (independently of $M$) by the Kazhdan number $K(\Gamma)$ of $\Gamma$.
\end{thm}

\begin{proof}[Proof of Proposition \ref{prop:infgenericprop(T)}]
We construct a chain of II$_1$ factors
\[
M_0 \subseteq N_0 \subseteq M_1 \subseteq N_1 \subseteq \dots
\]
inductively as follows.  Let $M_0$ be an infinitely generic II$_1$ factor.  Once $M_i$ is chosen, let $N_i$ be a property (T) factor constructed using Theorem \ref{propTfact}, whence its Kazhdan number is bounded by $K(\Gamma)$.  Once $N_i$ is chosen, let $M_{i+1}$ be an infinitely generic II$_1$ factor containing $N_i$, which exists because infinitely generic II$_1$ factors are embedding universal.

Let $M$ be the inductive limit of the chain.  Then $M$ has uniform local property (T) by construction.  It is also infinitely generic by \cite[Lemma 6.3.16]{AGKE-GeneralizedJung}.
\end{proof}

\begin{thm} \label{thm:mainthm2}
Every infinitely generic II$_1$ factor is uniformly super McDuff.  Consequently, any II$_1$ factor that is elementarily equivalent to an infinitely generic II$_1$ factor is uniformly super McDuff.
\end{thm}

\begin{proof}
As noted in the preliminaries, any two infinitely generic II$_1$ factors are elementarily equivalent.  Since the uniform super McDuff property is preserved by elementary equivalence, it suffices to show that \emph{some} infinitely generic II$_1$ factor is uniformly super McDuff.  By Proposition \ref{prop:infgenericprop(T)}, there exists an infinitely generic II$_1$ factor with uniform local property (T), which must be uniformly super McDuff by Proposition \ref{uniformpropT}.
\end{proof}

Theorem \ref{thm:mainthm2} generalizes Chifan, Drimbe, and Ioana's result, which showed that infinitely generic II$_1$ factors are super McDuff \cite[Theorem 6.4]{ChifanDI-EmbUnivPptyT}.  In fact, it is a strict generalization because there are II$_1$ factors elementarily equivalent to the infinitely generic ones without themselves being infinitely generic (Lemma \ref{infgenfact}).

% The following lemma will be used several times throughout the paper:

% \begin{lem}\label{propTlemma}
% For any $K\geq 1$, there is a function $\alpha_K:(0,1)\to (0,1)$ with the following property:  Suppose that $N$ has property (T) with finite Kazhdan set $G$ and Kazhdan constant at most $K$.  Then for any II$_1$ factor $M$ containing $N$ and any $(p,q)\in \cal P_2(M)$ with $\max_{x\in G}\max(\|[x,p]\|_2,\|[x,q]\|_2)<\alpha_K(\epsilon)$, there is $(p',q')\in \cal P_2(N'\cap M)$ with $\|p-p'\|_2,\|q-q'\|_2<\epsilon$.
% \end{lem}

% \begin{proof}
% Suppose that $\delta<\frac{1}{K}$ and $(p,q)\in \cal P_2(M)$ is such that $$\max_{x\in G}\max(\|[x,p]\|_2,\|[x,q]\|_2)<\delta.$$  Then by the definition of Kazhdan set, there is $(p_0,q_0)\in N'\cap M$ such that $\|p-p_0\|_2,\|q-q_0\|_2<K\delta$.  Let $\beta:\bb R^{>0}\to \bb R^{>0}$ be the uniform continuity modulus for $\xi$ and $\gamma:\bb R^{>0}\to \bb R^{>0}$ be the definability modulus for $\xi$.  Then if $K\delta<\beta(\gamma(\frac{\epsilon}{2}))$, we have $\xi(p_0,q_0)<\gamma(\frac{\epsilon}{2})$, whence there is $(p',q')\in \cal P_2(N'\cap M)$ such that $\|p_0-p'\|_2,\|q_0-q'\|_2<\frac{\epsilon}{2}$.  Thus, $\alpha_K(\epsilon):=\min(\frac{1}{K}\beta(\gamma(\frac{\epsilon}{2})),\frac{\epsilon}{2K})$ is the desired function.
% \end{proof}

Theorem \ref{thm:mainthm2} also allows us to connect two open questions discussed in Subsection \ref{subsection-e.c.}: whether every e.c. factor is super McDuff and whether the union of a chain of super McDuff factors is once again super McDuff.

\begin{prop}
If the union of every chain of uniformly super McDuff factors is (uniformly) super McDuff, then every e.c. factor is (uniformly) super McDuff.
\end{prop}

\begin{proof}
    Let $P$ be an e.c. factor and embed $P$ into an infinitely generic factor $M$.  By Theorem \ref{thm:mainthm2}, $M$ is uniformly super McDuff.  Fix an embedding $M\hookrightarrow P^\u$ that restricts to the diagonal embedding $P\hookrightarrow P^\u$.  We then have the chain
    $$P\subseteq M\hookrightarrow P^\u \subseteq M^\u \hookrightarrow P^{2\u}\subseteq M^{2\u}\cdots$$
Note that all the embeddings $P^{k\u}\hookrightarrow P^{(k+1)\u}$ are elementary (being iterated ultrapowers of the diagonal embedding).  Consequently, the limit $Q$ of the chain is an elementary extension of $P$.  Since $M$ is uniformly super McDuff and each $M^{k\u}$ is elementarily equivalent to $M$, they are also uniformly super McDuff.  By hypothesis, $Q$ is (uniformly) super McDuff.  Since (uniform) super McDuffness passes to elementary substructures, the result follows.
\end{proof}

\subsection{Strongly McDuff implies uniformly super McDuff} \label{subsec:stronglyMcDuff}

The following result was suggested to us by Adrian Ioana. 

\begin{thm} \label{thm:RUSM}
Let $M$ be a full tracial factor.  Then $M \otimes \R$ has the FC limit spectral gap property and hence is uniformly super McDuff.
\end{thm}

In particular, by taking $M = \bC$, we obtain a direct argument that $\R$ is uniformly super McDuff (which was already known because $\R$ has the Brown property).  When $M$ is a full II$_1$ factor, $M\otimes \R$ is often called a \textbf{strongly McDuff} II$_1$ factor.  Strongly McDuff II$_1$ factors were shown to be super McDuff in \cite[Proposition 6.1.5]{AGKE-GeneralizedJung} and Theorem \ref{thm:RUSM} generalizes this latter result to obtain that strongly McDuff factors are uniformly super McDuff.  As shown in \cite[Proposition 6.3.2]{AGKE-GeneralizedJung}, e.c. factors are never strongly McDuff, whence Theorem \ref{thm:RUSM} represents a completely different class of examples of uniformly super McDuff factors from the infinitely generic factors above.

The main ingredient in the proof of Theorem \ref{thm:RUSM} is the following result of Hastings on quantum expanders \cite{Hastings2007}.  Here we only state a special case of Hastings' result; in Hastings' terminology, this is the Hermitian case with $D = 4$.

\begin{thm}[Hastings {\cite{Hastings2007}}]
Fix an even $d \geq 4$.  There exist unitaries $U_{k,1}$, \dots, $U_{k,d}$ in $M_k(\bC)$ such that the following properties hold:  Let $\Phi_k: M_k(\bC) \to M_k(\bC)$ be defined by
\[
\Phi_k(X) = \frac{1}{d} \sum_{j=1}^d U_{k,j}XU_{k,j}^*.
\]
Let $\lambda_2(\Phi_k)$ be the second largest eigenvalue of $\Phi_k$, which is the largest eigenvalue of $\Phi_k$ on the orthogonal complement of $\bC 1$.  Then
\[
\lim_{k \to \infty} \lambda_2(\Phi_k) = \frac{2\sqrt{d-1}}{d}.
\]
\end{thm}

For us, the importance of Hastings' result is the following:

\begin{cor}
The spectral gap numbers for $M_k(\bC)$ satisfy $\limsup_{k \to \infty} \operatorname{SG}(M_k(\bC)) \leq 8$. In particular, $\sup_k \operatorname{SG}(M_k(\bC)) < \infty$.
\end{cor}

\begin{proof}
Fix an even $d \geq 4$ (in the end, we will use $d = 8$), and let $\Phi_k$ be as in the previous theorem.  By the definition of $\lambda_2(\Phi_k)$, we have
\[
\norm{X - \tr_k(X)1}_2 \leq \frac{1}{1-\lambda_2(\Phi_k)} \norm{X - \Phi_k(X)}_2.
\]
We also have
\[
\norm{X - \Phi_k(X)}_2
= \norm{X - \sum_{j=1}^d U_{k,j}XU_{k,j}^*}_2
\leq \frac{1}{d} \sum_{j=1}^d \norm{X - U_{k,j}XU_{k,j}^*}_2
= \frac{1}{d} \sum_{j=1}^d \norm{U_{k,j} X - X U_{k,j}}_2.
\]
Thus, by the Cauchy-Schwarz inequality, we have
\begin{align*}
\norm{X - \tr_k(X)1}_2^2 &\leq \frac{1}{(1 - \lambda_2(\Phi_k))^2} \left( \frac{1}{d} \sum_{j=1}^d \norm{U_{k,j} X - X U_{k,j}}_2 \right)^2 \\
&\leq \frac{1}{(1 - \lambda_2(\Phi_k))^2} \frac{1}{d} \sum_{j=1}^d \norm{U_{k,j}X - XU_{k,j}}_2^2.
\end{align*}
Therefore,
\[
\limsup_{k \to \infty} \operatorname{SG}(M_k(\bC)) \leq \limsup_{k \to \infty} \max\left(d, \frac{1}{4(1 - \lambda_2(\Phi_k))^2} \right) = \max\left(d, \frac{1}{d(1 - 2 \sqrt{d-1}/d)^2} \right).
\]
The second term inside the maximum equals $d(d + 2 \sqrt{d-1})/(d-2)^2$, which is less than $d$ for $d \geq 8$, and hence by taking $d = 8$ we get the result of the corollary.
\end{proof}

We will also use the following elementary observations about spectral gap numbers and tensor products.  These results are known to experts, but we include the proof here for completeness (and because we pay attention to the spectral gap numbers).

\begin{lem} \label{lem:tensorproduct}
Let $M$ be a full tracial factor.  Then for any tracial von Neumann algebra $N$, $M \otimes \bC$ has spectral gap inside $M \otimes N$.  In fact, if $F$ and $C$ witness spectral gap for $M$ inside $M$, then $F \otimes 1$ and $C$ witness spectral gap for $M \otimes \bC$ inside $M \otimes N$.  In particular, $\operatorname{SG}(M,M\otimes N) \leq \operatorname{SG}(M,M)$.
\end{lem}

\begin{proof}
Let $S$ be the operator $\id - \tau$ on $L^2(M)$ and let $T:L^2(M) \to L^2(M)^F$ be the linear map given by $T(\xi)= (x\xi - \xi x)_{x \in F}$.  Then the spectral gap property of $M\subseteq M$ can be expressed by
\[
\norm{S\xi}^2 \leq C \norm{T\xi}^2 \text{ for } \xi \in L^2(M),
\]
or equivalently, $S^*S \leq C \cdot T^*T$.  Note that $L^2(M \otimes N) \cong L^2(M) \otimes L^2(N)$ and $E_{\bC \otimes N} = \tau_M \otimes \id_{L^2(N)}$ as an operator on $L^2(M) \otimes L^2(N)$.  Since $S^*S \leq C \cdot T^*T$, we have
\[
(S \otimes \id_{L^2(N)})^*(S \otimes \id_{L^2(N)}) \leq C \cdot (T \otimes \id_{L^2(N)})^*(T \otimes \id_{L^2(N)}),
\]
which translates to
\[
\norm{\xi - E_{\bC \otimes N}(\xi)}_2^2 \leq C \sum_{x \in F} \norm{(x \otimes 1) \xi - \xi (x \otimes 1)}_2^2,
\]
yielding the desired result.
\end{proof}

\begin{lem} \label{lem:tensorproduct2}
Let $M_1$ and $M_2$ be full tracial factors.  Then $M_1 \otimes M_2$ is also full.  More precisely, for $j=1,2$, let $F_j$ and $C_j$ witness the spectral gap of $M_j\subseteq M_j$.  Then $(F_1 \otimes 1) \cup (1 \otimes F_2)$ and $\max(C_1,C_2)$ witness spectral gap for $M_1 \otimes M_2$.  Hence, $\operatorname{SG}(M_1 \otimes M_2) \leq \operatorname{SG}(M_1) + \operatorname{SG}(M_2)$.
\end{lem}

\begin{proof}
Fix $\xi \in L^2(M_1 \otimes M_2)$ and set $1 \otimes \eta = E_{\bC \otimes M_2}(\xi)$.  By the previous lemma,
\[
\norm{\xi - 1 \otimes \eta}_2^2 \leq C_1 \sum_{x \in F_1} \norm{(x \otimes 1) \xi - \xi (x \otimes 1)}_2^2.
\]
By the spectral gap of $M_2\subseteq M_2$,
\begin{align*}
\norm{\eta - \tau_N(\eta)}_2^2 &\leq C_2 \sum_{y \in F_2} \norm{y \eta - \eta y}_2^2 = C_2 \sum_{y \in F_2} \norm{(1 \otimes y)(1 \otimes \eta) - (1 \otimes \eta)(1 \otimes y)}_2^2 \\
&= C_2 \sum_{y \in F_2} \norm{(1 \otimes y)E_{1 \otimes M_2}(\xi) - E_{\bC \otimes M_2}(\xi)(1 \otimes y)}_2^2 \\
&= C_2 \sum_{y \in F_2} \norm{E_{\bC \otimes M_2}((1 \otimes y)\xi) - E_{\bC \otimes M_2}(\xi(1 \otimes y))}_2^2 \\
&\leq C_2 \sum_{y \in F_2} \norm{(1 \otimes y)\xi - \xi(1 \otimes y)}_2^2.
\end{align*}
Now $\xi - 1 \otimes \eta$ is orthogonal to $\bC \otimes M_2$ and in particular to $1 \otimes \eta - \tau_{M_2}(\eta)$, hence
\begin{align*}
\norm{\xi - \tau_{M_1 \otimes M_2}(\xi)}_2^2 &= \norm{\xi - E_{\bC \otimes M_2}(\xi)}_2^2 + \norm{\eta - \tau_{M_2}(\eta)}_2^2 \\
&\leq C_1 \sum_{x \in F_1} \norm{(x \otimes 1) \xi - \xi (x \otimes 1)}_2^2 + C_2 \sum_{y \in F_2} \norm{(1 \otimes y)\xi - \xi(1 \otimes y)}_2^2 \\
&\leq \max(C_1,C_2) \left( \sum_{x \in F_1} \norm{(x \otimes 1) \xi - \xi (x \otimes 1)}_2^2 + \sum_{y \in F_2} \norm{(1 \otimes y)\xi - \xi(1 \otimes y)}_2^2 \right).
\end{align*}

This concludes the proof of the Lemma. 
\end{proof}

We are now ready to prove Theorem \ref{thm:RUSM}.

\begin{proof}[Proof of Theorem \ref{thm:RUSM}]
Let $F \subseteq M \otimes \R$ be a finite subset and let $\epsilon > 0$.  Set $s = \operatorname{SG}(M) + \sup_k \operatorname{SG}(M_k(\bC))$.  Then there exists some tensor decomposition $\R = M_k(\bC) \otimes \tilde{\R}$ such that $F$ is contained in the $\epsilon$-neighborhood of $M \otimes M_k(\bC) \otimes \bC$.  Note also that
\[
(M \otimes M_k(\bC) \otimes \bC)' \cap (M \otimes \R) = \bC \otimes \bC \otimes \tilde{\R},
\]
which is a factor.  By Lemmas \ref{lem:tensorproduct} and \ref{lem:tensorproduct2},
\[
\operatorname{SG}(M \otimes M_k(\bC) \otimes \bC,M \otimes \R) \leq \operatorname{SG}(M \otimes M_k(\bC)) \leq \operatorname{SG}(M) + \operatorname{SG}(M_k(\bC)) \leq s.
\]
Therefore, $M \otimes \R$ has the FC limit spectral gap property, as desired.
\end{proof}

In \cite[Proposition 6.1.7]{AGKE-GeneralizedJung}, it was shown that an infinite tensor product of full II$_1$ factors is super McDuff.  It is natural to wonder whether or not there is a uniform version of this result.  The proof of Theorem \ref{thm:RUSM} immediately yields the following:

\begin{prop}\label{uniformfull}
Suppose that $(M_k)$ is a sequence of full II$_1$ factors for which $$\sup_k \operatorname{SG}(M_1\otimes\cdots\otimes M_k)<\infty.$$  Then $\bigotimes_k M_k$ is uniformly super McDuff.
\end{prop}

\begin{remark}
It is unknown whether or not there is a uniform bound on the spectral gap number of  full II$_1$ factors or, equivalently, whether or not the class of full II$_1$ factors is closed under ultraproducts.  Proposition \ref{uniformfull} shows that a positive answer to this question would imply that an infinite tensor product of full factors is uniformly super McDuff.  Said differently, if one can find an example of an infinite tensor product of full II$_1$ factors which is not uniformly super McDuff, then one can conclude that the spectral gap numbers for full II$_1$ factors is unbounded.
\end{remark} 

\section{Do finitely generic factors have the uniform super McDuff property?}

The results concerning the infinitely generic factors above motivate the following questions for the finitely generic factors:

\begin{question}
Is there a finitely generic (uniformly) super McDuff factor?  Are all finitely generic factors (uniformly) super McDuff?
\end{question}

\begin{remark}
If one restricts to the case of Connes embeddable factors, then $\R$ is the unique finitely generic factor (see \cite[Corollary 3.2.6]{AGKE-GeneralizedJung}).  Let $(\sigma_n)$ be a sequence of sentences witnessing that $\R$ is uniformly super McDuff.  By \cite[Theorem 2.14 and Proposition 3.9]{Goldbring-Enforceable}, for each $n\geq 1$, there is a unique real number $r_n$ such that $\sigma_n^M=r_n$ in all finitely generic factors $M$.  If $r_n>0$ for some $n$, then we conclude that the finitely generic factors are not uniformly super McDuff.  This result might lead to a potentially different proof that the Connes Embedding Problem has a negative solution (see \cite{MIP*eqsRE}).   
\end{remark}

% \begin{prop}
% Suppose that $P$ is finitely generic and there is a II$_1$ factor $N\supseteq P$ for which there is an embedding $N\hookrightarrow P^\u$ with factorial relative commutant.  Then $P$ is super McDuff.  
% \end{prop}

% \begin{proof}
% Since $P$ is finitely generic, we have that the induced embedding $i:P\subseteq N \hookrightarrow P^\u$ is elementary.  Then by \cite[Theorem 3.2.2.]{AGKE-GeneralizedJung}, we have that $i(P)'\cap P^\u$ has factorial commutant.
% \end{proof}

\subsection{A sufficient criterion for finitely generic factors to be uniformly super McDuff}

In this subsection, we present a sufficient criterion for a finitely generic factor to be uniformly super McDuff.  We first need a preliminary lemma.

\begin{lem}\label{propTlemma}
 For any $K\geq 1$, there is a function $\alpha_K:(0,1)\to (0,1)$ with the following property:  Suppose that $N$ has property (T) with finite Kazhdan set $G$ and Kazhdan constant at most $K$.  Then for any II$_1$ factor $M$ containing $N$ and any $(p,q)\in \cal P_2(M)$ with $\max_{x\in G}\max(\|[x,p]\|_2,\|[x,q]\|_2)<\alpha_K(\epsilon)$, there is $(p',q')\in \cal P_2(N'\cap M)$ with $\|p-p'\|_2,\|q-q'\|_2<\epsilon$.
\end{lem}

 \begin{proof}
 Suppose that $\delta<\frac{1}{K}$ and $(p,q)\in \cal P_2(M)$ is such that $$\max_{x\in G}\max(\|[x,p]\|_2,\|[x,q]\|_2)<\delta.$$  Then by the definition of Kazhdan set, there is $(p_0,q_0)\in N'\cap M$ such that $\|p-p_0\|_2,\|q-q_0\|_2<K\delta$.  Recall from Subsection \ref{subsection:definability} the formula $\xi$ whose zeroset is $\cal P_2$.  Let $\beta:\bb R^{>0}\to \bb R^{>0}$ be the uniform continuity modulus for $\xi$ and $\gamma:\bb R^{>0}\to \bb R^{>0}$ be the definability modulus for $\xi$.  Then if $K\delta<\beta(\gamma(\frac{\epsilon}{2}))$, we have $\xi(p_0,q_0)<\gamma(\frac{\epsilon}{2})$, whence there is $(p',q')\in \cal P_2(N'\cap M)$ such that $\|p_0-p'\|_2,\|q_0-q'\|_2<\frac{\epsilon}{2}$.  Thus, $\alpha_K(\epsilon):=\min(\frac{1}{K}\beta(\gamma(\frac{\epsilon}{2})),\frac{\epsilon}{2K})$ is the desired function.
 \end{proof}

\begin{prop}\label{fingenpopa}
Suppose that $P$ is finitely generic and there is a II$_1$ factor $N\supseteq P$ \emph{with property (T)} for which there is an embedding $N\hookrightarrow P^\u$ with factorial relative commutant.  Then $P$ is uniformly super McDuff. 
\end{prop}

\begin{proof}
We consider the sequence of embeddings
$$P\subseteq N\hookrightarrow P^\u\subseteq N^\u\hookrightarrow P^{2\u}\subseteq  N^{2\u} \cdots,$$ where for each $k\in \bb N$, we let $P^{k\u}$ denote the $k^{\text{th}}$-iterated ultrapower of $P$ and similarly for other iterated ultrapowers and the embeddings $N^{k\u}\hookrightarrow P^{(k+1)\u}$ are induced by the original embedding $N\hookrightarrow P^\u$.  We let $Q$ denote the union of the chain.  Since $P$ is finitely generic, we have that the embedding $P\hookrightarrow P^\u$ is elementary, whence so are all of the iterated embeddings $P^{k\u}\hookrightarrow P^{(k+1)\u}$, and thus $Q\equiv P$.  It thus suffices to show that $Q$ is uniformly super McDuff.

Let $G$ be a finite Kazhdan set for $N$ with Kazhdan constant $K$.  

\noindent \textbf{Claim:}  Fix $\epsilon>0$.  For all $k\geq 1$ , and all $(p,q)\in \cal P_2(P^{k\u})$, if $\|[x,p]\|_2,\|[x,q]\|_2<\alpha_K(\frac{\epsilon}{2})$ for all $x\in G$ (viewed as elements of $P^{k\u}$ in the obvious way), then there is $u\in U((N^{(k-1)\u})'\cap P^{k\u})$ such that $\|upu^*-q\|_2\leq \epsilon$.

We prove the claim by induction on $k$.  First suppose that $k=1$. If 
$$\max_{x\in G}\max(\|[x,p]\|_2,\|[x,q]\|_2)<\alpha_K \left(\frac{\epsilon}{2} \right),$$
then there is $(p',q')\in \cal P_2(N'\cap P^\u)$ with $\|p-p'\|_2,\|q-q'\|_2<\frac{\epsilon}{2}$.  Since $N'\cap P^\u$ is a factor, there is $u\in U(N'\cap P^\u)$ with $up'u^*=q'$.  It follows that $\|upu^*-q\|_2\leq \epsilon$.  Thus, the case $k=1$ is proven.

% For suitably chosen $\delta$, we have, like in the proof of Proposition \ref{uniformpropT}, that $(p,q)$ is sufficiently close to $(p',q')\in \cal P_2(N'\cap P^\u)$.  Since $N'\cap P^\u$ is a factor, there is $u\in U(N'\cap P^\u)$ with $up'u^*=q'$.  It follows that $\|upu^*-q\|_2\leq \epsilon$.  Thus, the case $k=1$ is proven.

We now prove the inductive step.  Suppose the result is true for $k$ and consider $(p,q)\in \cal P_2(P^{(k+1)\u})$ with $\|[x,p]\|_2,\|[x,q]\|_2<\alpha_K(\frac{\epsilon}{2})$ for all $x\in G$.  Write $p=(p_n)_\u$ and $q=(q_n)_\u$ with $(p_n,q_n)\in \cal P_2(P^{k\u})$.  It follows that for $\u$-almost all $n$ we have $\|[x,p_n]\|_2,\|[x,q_n]\|_2<\alpha_K(\frac{\epsilon}{2})$ for all $x\in G$, whence, by induction, for these $n$, there are $u_n\in U((N^{(k-1)\u})'\cap P^{k\u})$ with $\|u_np_nu_n^*-q_n\|_2\leq \epsilon$.  Letting $u=(u_n)_\u\in U((N^{k\u})'\cap P^{(k+1)\u})$, we see that $\|upu^*-q\|_2\leq \epsilon$, as desired.  The claim is thus proven.

We now verify that $Q$ is uniformly super McDuff.  Fix $\epsilon>0$ and take $\delta$ as in the claim for $\epsilon$.  Fix $F\subseteq Q$ finite.  Take $k\geq 1$ such that $F\subseteq_{\frac{\epsilon}{2}} N^{(k-1)^\u}$.  Now suppose that $(p,q)\in \cal P_2(Q)$ are such that $\max_{x\in G}\max(\|[x,p]\|_2,\|[x,q]\|_2)<\rho$ (for $\rho>0$ to be determined).  Take $l\geq k$ such that there are $(p',q')\in \cal P_2(P^{l\u})$ with $\|p-p'\|_2,\|q-q'\|_2<\rho$. Then $\max_{x\in G}\max(\|[x,p']\|_2,\|[x,q']\|_2)<\alpha_K(\frac{\epsilon}{4})$ if $\rho$ was small enough, whence by the claim there is $u\in U((N^{(l-1)\u})'\cap P^{l\u})$ with $\|up'u^*-q'\|_2\leq \frac{\epsilon}{2}$.  It follows that $\|upu^*-q\|_2<\epsilon$ and $\|[u,x]\|_2<\epsilon$ for all $x\in F$.
\end{proof}

\begin{remark}
We note that, in the previous proposition, one can obtain the weaker conclusion that $P$ is super McDuff using \cite[Theorem 3.2.2.]{AGKE-GeneralizedJung}.
\end{remark}

\subsection{The (T)-FRC property}  To apply Proposition \ref{fingenpopa} to a finitely generic $P$, we would need a property (T) factor $N$ containing $P$ which admits an embedding into $P^{\u}$ with factorial relative commutant.  In fact, one can ask whether every property (T) factor $N$ admits such an embedding, motivating the following definition.

\begin{defn}
We say that a separable II$_1$ factor $M$ has the \textbf{(T)-FRC property} if every II$_1$ factor with property (T) admits an embedding into $M^\u$ with factorial relative commutant.
\end{defn}

It will be a consequence of Proposition \ref{popaaxiom} below that whether or not a given property (T) factor admits an embedding into $M^\u$ with factorial relative commutant is independent of the choice of ultrafilter, making the (T)-FRC well-defined.  Another consequence of Proposition \ref{popaaxiom} is that the (T)-FRC property is preserved under elementary equivalence.

We may also relativize the (T)-FRC property to the setting of Connes embeddable factors:

\begin{defn}
A Connes embeddable separable II$_1$ factor $M$ has the \textbf{embeddable (T)-FRC property} if every Connes embeddable property (T) factor admits an embedding into $M^\u$ with factorial relative commutant. 
\end{defn}

A well-known question of Popa asks whether or not $\R$ has the embeddable (T)-FRC property.

% The Popa property is named after a problem of Popa, who, in this terminology, asks if $\R$ has the embeddable Popa property.  

Proposition \ref{fingenpopa} together with embedding universality of Property (T) factors yields the following:

\begin{cor}\label{fingenPopaBrown}
If the finitely generic factors have the (T)-FRC property, then they are also uniformly super McDuff.
\end{cor}

\begin{cor}
If the (T)-FRC property is enforceable, then being uniformly super McDuff is enforceable.  In particular, if the enforceable factor exists and has the (T)-FRC property, then it is uniformly super McDuff.
\end{cor}
% \begin{proof}
% Being finitely generic is enforceable.
% \end{proof}

\begin{remarks}

\

\begin{enumerate}

\item The conclusion of Corollary \ref{fingenPopaBrown} holds in the Connes embeddable case as well.  As mentioned above, $\R$ is the unique separable finitely generic Connes embeddable factor.  Thus, if $\R$ has the embeddable (T)-FRC property, then this would give yet a different argument that $\R$ is uniformly super McDuff (which, to be fair, is less informative than both Brown's original proof that $\R$ has the Brown property, which tells one what the intermediate subfactor is that has factorial commutant, and our proof above using quantum expanders).
\item It is interesting to note that in the Connes embeddable case, the finitely generic factor is uniformly super McDuff although we do not know if the infinitely generic factors are uniformly super McDuff.  (It is an open question as to whether or not $\R$ is an infinitely generic embeddable factor.)  On the other hand, in the unrestricted case, as shown in Theorem \ref{thm:mainthm2}, the infinitely generic factors are uniformly super McDuff and we are unsure whether the finitely generic factors are uniformly super McDuff.
\end{enumerate}
\end{remarks}

Given Corollary \ref{fingenPopaBrown}, it is now natural to ask:

\begin{question}
    Do finitely generic factors have the (T)-FRC property?
\end{question}

In \cite{Goldbring-PopaFCEP}, it was shown that the infinitely generic factors have the (T)-FRC property.  Thus, if the finitely and infinitely generic factors are elementarily equivalent, then the finitely generic factors also have the (T)-FRC property.

\subsection{Model-theoretic aspects of the (T)-FRC property}

% \begin{question}
%     Is the converse true? That is, if the enforceable factor exists and has the Brown property, then does it also have the Popa property?
% \end{question}

In this section, we explain why proving the (T)-FRC property for infinitely generic factors is easier than proving the (T)-FRC property for an arbitrary e.c. factor, which in turn is easier than proving the (T)-FRC property for finitely genric factors.  We first observe:

\begin{prop}\label{popaaxiom}
For any property (T) factor $N$, there is a family $(\sigma_{N,m})_{m\geq 1}$ of $\exists_3$-sentences such that the following are equivalent for any II$_1$ factor $M$:
\begin{enumerate}
    \item $\sigma^M_{N,m}=0$ for all $m\geq 1$.
    \item For every ultrafilter $\u$, $N$ embeds into $M^\u$ with factorial relative commutant.
    \item For some ultrafilter $\u$, $N$ embeds into $M^\u$ with factorial relative commutant.
\end{enumerate}
\end{prop}

\begin{proof}
Fix a property (T) factor $N$ with generating set $F$ of cardinality $n$ and Kazhdan constant $K$.  For each $m\geq 1$, consider the following formula $\psi_{N,m}(\vec x)$ in the free variables $\vec x=(x_1,\ldots,x_n)$:
\begin{multline*}
\sup_{(p,q)\in \cal P_2} \min \biggl( \alpha_K \left(\frac{1}{2m} \right) \dminus \max_{i=1,\ldots,n}\max(\|[x_i,p]\|_2,\|[x_i,q]\|_2),\\
(\inf_{u\in U}\max(\|upu^*-q\|_2,\max_{i=1,\ldots,n}\|[u,x_i]\|_2)\dminus \frac{1}{m}) \biggr) = 0,
\end{multline*}
where $\alpha_K$ is as in Lemma \ref{propTlemma}.
Enumerate a countable dense subset of $\operatorname{qftp}^N(F)$ as $\{\theta_{N,i}(\vec x) \ : \ i\in \omega\}$ and let $\sigma_{N,m}$ be $\inf_x\max(\max_{i=1,\ldots,m}\theta_i(x),\psi_{N,m}(x))$.  Note that $\psi_{N,m}(\vec x)$ is a $\forall_2$-formula, whence $\sigma_{N,m}$ is a $\exists_3$-sentence.  We claim that these are the desired sentences.

First suppose that $\sigma_{N,m}^M=0$ for all $m\geq 1$.  Then by saturation, there is a tuple $\vec a=(a_1,\ldots,a_n)$ in $M^\u$ satisfying $\operatorname{qftp}^N(F)$ and for which $\psi_{N,m}^{M^\u}(\vec a)=0$ for all $m\geq 1$.  There is thus an embedding $N\hookrightarrow M^\u$ obtained by sending the elements of $F$ to the $a_i$'s.  We claim that this embedding has factorial relative commutant.  To see this, fix $(p,q)\in \cal P_2(N'\cap M^\u)$ and $m\geq 1$.  Then the first part of the minimum appearing in $\psi_{N,m}$ is $\alpha_K(\frac{1}{2m})$, whence the second part of the minimum must be $0$.  Thus, there is $u\in U(M^\u)$ such that $\|upu^*-q\|_2\leq \frac{1}{m}$ and $\|[u,a_i]\|_2\leq \frac{1}{m}$.  Since such a unitary exists for all  $m\geq 1$, by saturation, there is $u\in U(N'\cap M^\u)$ conjugating $p$ to $q$, as desired.

Now suppose that $N$ admits an embedding into $M^\u$ with factorial relative commutant for some ultrafilter $\u$.  Let $a_1,\ldots,a_n\in M^\u$ be the images of the elements of $F$ under the embedding.  Fix $m\geq 1$.  We show that $\psi_{N,m}^{M^\u}(\vec a)=0$, whence $\sigma_{N,m}^M=\sigma_{N,m}^{M^\u}=0$.  Towards this end, fix $(p,q)\in \cal P_2(M^\u)$ and suppose $\|[a_i,p]\|_2,\|[a_i,q]\|_2<\alpha_K(\frac{1}{2m})$ for all $i=1,\ldots,n$.  Arguing exactly as in the proof of Proposition \ref{fingenpopa}, there is $u\in U(N'\cap M^\u)$ satisfying $\|upu^*-q\|_2\leq \frac{1}{m}$.  This $u$ realizes the inner infimum in $\psi_{N,m}$.
\end{proof}

\begin{cor}
The (T)-FRC property and the embeddable (T)-FRC property are both $\exists_3$-axiomatizable properties of separable II$_1$ factors.
\end{cor}

\begin{remarks}

\

\begin{enumerate}
    \item One could use the sentences in Proposition \ref{popaaxiom} to give a definition of the (T)-FRC property for arbitrary II$_1$ factors.
    \item It is worth pointing out that the (T)-FRC property is axiomatizable while we currently only know that the Brown property is merely \emph{local}, that is, preserved under elementary equivalence.
\end{enumerate}
\end{remarks}

Let us recall the following various facts about truth of $\exists_3$ sentences in e.c. structures, which, for simplicity, we state only for II$_1$ factors (although the first item holds in general and the second and third item hold in any $\forall_2$ theory with the joint embedding property).  The classical versions of these results appear as \cite[Corollary 3.2.5, Exercise 5.3.12, Theorem 4.3.4]{games}.

\begin{fact}
Suppose that $\sigma$ is a $\exists_3$-sentence in the language of tracial von Neumann algebras.
\begin{enumerate}
    \item If $M\subseteq N$ are e.c. factors and $\sigma^M=0$, then $\sigma^N=0$.  (Here, $\sigma$ can even have parameters from $M$.)
    \item  If $\sigma^M=0$ for some e.c. factor, then $\sigma^N=0$ for all infinitely generic factors.
    \item If $\sigma^M=0$ for some finitely generic factor, then $\sigma^N=0$ for all e.c. factors.
\end{enumerate}
\end{fact}

\begin{proof}[Proof of (1) and (2)]
For (1), fix an embedding $i: N\hookrightarrow M^\u$ restricting to the diagonal embedding $M\hookrightarrow M^\u$.  Write $\sigma=\inf_x\sup_y\inf_z \varphi(x,y,z)$, where $\varphi$ is a quantifier-free formula with parameters from $M$.  Since $\sigma^M=0$, there is $a\in M$ such that $(\sup_y\inf_z\varphi(a,y,z))^{M}=0$.  Fix $b\in N$; we need $(\inf_z\varphi(a,b,z))^N=0$.  However, since $N$ is e.c., this follows from $(\inf_z \varphi(a,i(b),z))^{M^\u}=0$, which in turn follows from the fact that $(\sup_y\inf_z \varphi(a,y,z))^{M^\u}=0$.

(2) follows from (1) together with the facts that the infinitely generic factors are e.c., embedding universal, and are all elementarily equivalent to one another.
\end{proof}

The third item in the above proposition can be proven just as in the classical case, using the continuous versions of the necessary ingredients which were established in \cite{Goldbring-Enforceable}.

The second item in the previous fact explains why proving that the infinitely generic factors have the (T)-FRC property was the ``easiest'' possible while the third item explains why proving that $\R$ has the embeddable (T)-FRC property and the finitely generic factors have the (T)-FRC property are as ``hard'' as possible.

Item (1) of the above fact also yields the following fact:

\begin{cor}
If $M$ is an e.c. factor that contains an infinitely generic factor, then $M$ has the (T)-FRC property.
\end{cor}

We remark that it is still possible that all e.c. factors may be infinitely generic, whence the previous corollary would degenerate to the main result of \cite{Goldbring-PopaFCEP}.

\printbibliography

\end{document}